\documentclass[12 pt,twoside]{amsart}

\usepackage{eucal}
\usepackage{exscale}
\usepackage[arrow, matrix, curve]{xy}
\usepackage{amsmath}
\usepackage{amsfonts}
\usepackage{amsthm}
\usepackage{amssymb}
\usepackage[latin1]{inputenc}
\newtheorem{theorem}{Theorem}[section]
\newtheorem{lemma}[theorem]{Lemma}
\newtheorem{proposition}[theorem]{Proposition}
\newtheorem{corollary}[theorem]{Corollary}

\theoremstyle{definition}
\newtheorem{definition}[theorem]{Definition}
\newtheorem{example}[theorem]{Example}

\theoremstyle{remark}
\newtheorem{remark}[theorem]{Remark}

\numberwithin{equation}{section}

\begin{document}

\title[ The spectrum of operators on $C(K)$  ]{The spectrum of operators on $C(K)$ with the Grothendieck property and characterization of $J$-class operators which are adjoints}

\author[Amir Bahman Nasseri]{Amir Bahman Nasseri}
\address{Fachbereich C-Mathematik und Naturwissenschaften, Arbeitsgruppe Funktionalanalysis, Bergische Universit\"{a}t Wuppertal, D-42119 Wuppertal, Germany}
\email{nasseri@math.uni-wuppertal.de}

\date{}

\begin{abstract} This article deals with properties of spectra of operators on $C(K)$-spaces with the Grothendieck property (e.g. $l^{\infty}$) and application to so called $J$-class operators introduced by A. Manoussos and G. Costakis. We will show that $C(K)$ has the Grothendieck property if and only if the boundary of the spectrum of every operator on $C(K)$ consists entirely of eigenvalues of its adjoint. As a consequence we will see that there does not exist invertible $J$-class operators on  $C(K)$ with the Grothendieck property. In the third section we will give a quantitative and qualitative characterization of all $J$-class operators on $l^\infty$ which are adjoints from operators on $l^1$.
\end{abstract}

\maketitle

\section*{Preliminaries and notations} Let $X, Y$ be complex Banach spaces and denote by $L(X,Y)$ the Banach space of all bounded and linear operators from $X$ to $Y$. If $T:X\rightarrow X$ is a bounded linear operator then $\sigma(T)$ stands for the spectrum of $T$ and $\rho(T)$ for its resolvent set. An operator $T$ is called weakly compact if $T(B_{X})$ is relative weakly compact in $Y$. By $W(X)$ we denote set of all weakly compact operators on $X$ into $X$ and $L(X)/W(X)$ is called the weak Calkin algebra endowed with the usual quotient norm. An operator $T:X\rightarrow Y$ is called tauberian if $(T^{**})^{-1}(Y)\subset X$ holds. Further the operator $T^{co}: X^{**}/X\rightarrow Y^{**}/Y$ is defined by $T^{co}(x^{**}+X)=T^{**}x^{**}+Y$. For properties of $T^{co}$ and tauberian operators we refer the reader to \cite{M.G.}. Let $E\subset X$ be a $T$-invariant closed subspace of $T\in L(X)$, then the operator $\widehat{T}:X/E\rightarrow X/E$ is defined by $\widehat{T}\left[x\right]_{E}:=\left[Tx\right]_{E}$. If $A$ is a Banach algebra then $G(A)$ stands for the invertible elements and $G_{r}(A)\ (G_{l}(A))$ denotes the set of the right (left) invertible elements. An element $a\in A$ is called left (right) topological divisor of zero if there exists a norm one sequence $x_n$ in $A$ such that $ax_n\ (x_{n}a)$ tends to zero in norm. By $\Phi(X)$ we denote the set of the Fredholm-operators and $\Phi_{+}(X)$ stands for the set of all upper-semi Fredholm operators (i.e. finite dimensional kernel and closed range). Further $\Phi_{r}(X)$ stands for the set of all essentially right invertible operators ($T\in L(X)$ is called right essentially invertible if there exists $S\in L(X)$ and $K\in K(X)$ such that $TS=I+K$ holds). The symbol $\text{orb}(T,x)$ denotes the orbit of $x$ under $T$, i.e. $\text{orb}(T,x):=\{T^nx: n\in\mathbb{N}\}$. If $X$ is separable and $\text{orb}(T,x)$ is dense, then $T$ is called hypercyclic, which is equivalent to say that $T$ is topologically transitive, i.e. for each pair of non empty open subsets $U,V\subset X$ there exists a positive integer $n$, such that $T^{n}(U)\cap V\neq\emptyset$. By $J_{T}(x)$ we mean the $J$-set of $x$ under $T$, i.e.
\begin{align*}J_T(x):=\{y\in X: \text{ there exists a strictly increasing sequence}&\\ \text{of natural numbers }(k_n) \text{ and a sequence }&\\ (x_n)\text{ in } X, \text{ such that } x_n\rightarrow x \text{ and } T^{k_n}x_n\rightarrow y\}.
\end{align*}
If $J_{T}(x)=X$ for some $x\in X\backslash\{0\}$, then $T$ is called a $J$-class operator. By $l^\infty$ we denote the space of bounded sequences endowed with the usual $\sup$-norm. In \cite{Bermudez} it is shown that on $l^{\infty}$ there does not exist any topological transitive operator using spectral properties of operators on $l^{\infty}$ which are also listed and proved in \cite{Bermudez}. On the other hand there exist $J$-class operators like the weighted backward shift $\lambda B:l^{\infty}\rightarrow l^{\infty}$, $\lambda B(x_1,x_2,\ldots):=(\lambda x_2,\lambda x_3,\ldots )$ for $\left|\lambda\right|>1$ (see \cite{G.C.} and \cite{G.C.2}).
%
%

\section{Introduction}
In \cite{Bermudez} it is shown by T. Berm\'{u}dez and N. J. Kalton that the point spectrum of the adjoint of an operator defined on a $C(K)$-space with the Grothendieck property is non-empty. As a direct consequence there does not exists topological transitive operators on this space. In \cite{Lotz} H. Lotz obtained a similar result for operators on Banach spaces $X$ which have the Grothendieck- and the Dunford Pettis property but with restrictive assumptions on the resolvent $R(\lambda, T):=(\lambda I-T)^{-1}$ with $\lambda\in\rho(T)$.\\\\
\textbf{Theorem \textit{(Lotz, \cite{Lotz}).}} \textit{Let $X$ be a Banach space which has the\\ Grothendieck- and the Dunford Pettis property. Suppose $\lambda_n$ is a sequence in the resolvent set of $T\in L(X)$ with the property that $\lambda_n$ converges to some $\lambda\in\partial\sigma(T)$. Further assume that $(\lambda_n-\lambda)R(\lambda_n,T)$ is uniformly bounded. Then $\lambda$ is an eigenvalue of $T^{*}$.}\\\\
One can construct operators on $l^{\infty}$ for which the uniform boundedness condition in the above theorem is not satisfied. Indeed take a compact subset $K$ of $\mathbb{C}$ such that $K$ has a very sharp peak with the property that for every sequence $(\lambda_n)$ which converges outside $K$ to the peak, say $\lambda$, we have that $(\lambda_n-\lambda)\ (\text{dist}(\lambda_n,K))^{-1}$ is unbounded. Now take a diagonal operator $T$ on $l^{\infty}$ such that $\sigma(T)=K$. Then we obtain 
\[\left|\lambda_{n}-\lambda\right|\left\|R(\lambda_n,T)\right\|\geq (\lambda_n-\lambda)\ (\text{dist}(\lambda_n,K))^{-1}.\]
This shows that the uniform  boundedness condition is not satisfied. In this section we will show that on the space $C(K)$ with the Grothendieck property this condition is not necessary.
The main theorem in Section 2 reads as follows:\\\\
\textbf{Theorem.} \textit{Consider $T\in L(C(K))$, where $K$ is a compact Hausdorff space. Then $C(K)$ has the Grothendieck property (e.g. $C(\beta\mathbb{N})=l^{\infty}$ or $C(\beta\mathbb{N}\backslash\mathbb{N})=l^{\infty}/c_0$) if and only  
\[\partial\sigma(T)\subset\sigma_{p}(T^{*})\]
holds for every operator $T\in L(C(K))$. In both cases every $J$-class operator is not invertible.}\\\\
In Section 3 we will characterize the $J$-Class behavior for operators on $l^{\infty}$, which are adjoints of operators on $l^{1}$. As a corollary we get\\\\
\textbf{Corollary.} Consider the backward shift $B$ on $l^{\infty}$ and let $f$ be a holomorphic function in a neighborhood of the closed unit disc. Then the following statements are equivalent.\\\\
(i) $f(B)$ is $J$-class.\\\\
(ii) $f(\partial\mathbb{D})\cap\overline{\mathbb{D}}=\emptyset$ and $\overline{\mathbb{D}}\subseteq f(\mathbb{D})$.\\\\
It seems to be a hard task to determine for an operator $T$ on a Banach space X, those vectors such that $J_{T}(x)=X$. For $B_{w}:l^{\infty}\rightarrow l^\infty$, where $B_{w}$ is the backward shift and $w=(w_n)$ is a positive and bounded weight sequence, G. Costakis and A. Manoussos obtained that if $B_{w}$ is $J$-class then
\[A_{B_{w}}:=\{x\in l^{\infty} :\ J_{B_{w}}(x)=l^{\infty}\}=c_0.\]
We will see also in Section 3 that $f(B_w)$, where $f$ is a holomorphic function defined on a neighborhood of $B_w$ that
\[A_{f(B_{w})}:=\{x\in l^{\infty} :\ J_{f(B_{w})}(x)=l^{\infty}\}=c_0.\]
More generally we will see that if $T=S^{**}$, where $S:c_0\rightarrow c_0$ then we get $A_{T}\subset c_0$. In our last section we state some open questions occurred during our research.

%
%
\section{Spectrum of operators on $C(K)$-spaces with the Grothendieck property}
%
%
We begin with some lemmas.
\begin{lemma}\label{lemma-Weis} Let $M$ be a Banach space isomorphic to a complemented subspace of $L_1[0,1]$. Then we have \[\partial\Phi(M)\subset \{T\in L(M) : T \text{ is not tauberian}\}.\] 
\end{lemma}
\begin{proof}Consider first the case $M=L_1=L_1[0,1]$. By 3.1 in \cite{L.W.} the equation $\left\|T+W(L_1)\right\|=\left\|T^{co}\right\|$ holds for any $T\in L(L^1)$ (see also \cite{M.G.2}, proof of Thm. 2.2). Let $T+W(L_1)$ be right invertible in the weak Calkin algebra. Hence there exist operators $S\in L(L_1)$ and $W\in W(L_1)$
such that $TS=I+W$. Multiplying this equation from the right with $I-W$, we see that $T$ is essentially left invertible, since $L_1$ has Dunford Pettis Property and so $W^2$ is compact. It follows that the right invertible elements in the weak Calkin algebra are the same as the right essentialy invertible operators. Take $T\in\partial\Phi_{r}(L_1)$, then $[T]\in\partial G_{r}(A)$ with $A:=L(L_1)/W(L_1)$. One find now a sequence $T_n$ of operators such that $\left[T_n\right]=T_n+W(L_1)$ are right invertible and converge in norm to $[T]$. Hence there exists a sequence $S_n$ such that $[S_n]$ tends to infinity in norm and $[T_nS_n]=[I]$ for all $n$ (see \cite{V.M.}, p. 5, 6). Define now $\tilde{S}_n:=\frac{1}{\left\|\left[\tilde{S}_n\right]\right\|} S_n$. Then $1=\left\|\tilde{S}_n\right\|=\left\|(\tilde{S_n})^{co}\right\|$. From this we get
\begin{align*}\left\|T^{co}(\tilde{S_n})^{co}\right\|&=\left\|\left[T\right]\left[\tilde{S_n}\right]\right\|=\left\|\left[T_n+T-T_n\right]\frac{\left[S_n\right]}{\left\|\left[S_n\right]\right\|}\right\|\\&\leq\left\|\left[T-T_n\right]\right\|+\frac{1}{\left\|\left[S_n\right]\right\|}.
\end{align*}
This shows that $\left\|T^{co}(\tilde{S_n})^{co}\right\|$ tends to zero and therefore $T^{co}$ is a left topological divisor of zero, hence $T^{co}$ is not bounded below (see {V.M.}, p. 89) and therefore $T$ is not tauberian (see \cite{M.G.}, p. 89). Since $\partial\Phi\subset\partial\Phi_{r}$ (see \cite{V.M.} p. 5, 6) the statement follows for $L_1[0,1]$. Consider now the general case, where $M$ is isomorphic to some complemented subspace of $L_1[0,1]$. For our purpose it suffices to consider $M$ as a subspace of $L_1$ and hence we find a closed subspace $N$ of $L_1$ such that $M\oplus N=L_1$ holds. Take now an operator $T$ from $\partial\Phi(M)$ and a sequence $T_n$ in $\Phi(M)$ which converges to $T$. We can now extend the operator $T_n$ and $T$ by the identity on $N$ to the operators $\tilde{T_n}=I\oplus T_n$ and $\tilde{T}=I\oplus T$. Then $\tilde{T_n}$ is Fredholm and converges to $\tilde{T}\in\partial\Phi (L_1)$ and hence is not tauberian by the above. It is then easy to see that $T$ is not tauberian.
\end{proof}
%
%
The next lemma is essentially due to Tzafriri (see \cite{L.T.}, Lemma 1) who stated it for a single operator $T$ defined on $L_{1}(\Omega,\Sigma,\mu)$, where $(\Omega,\Sigma,\mu)$ is an arbitrary measure space.

\begin{lemma} Consider $L_1(\Omega,\Sigma,\mu)$ where $(\Omega,\Sigma,\mu)$ is an arbitrary measure space and let $\{T_n\}$ be any countable set of operators on $L_1(\Omega,\Sigma,\mu)$ into itself. Further let $M$ be any separable subspace of $L_1(\Omega,\Sigma,\mu)$. Then there exists a subset $\tilde{\Omega}$ of $\Omega$ and a $\sigma$-subring $\tilde{\Sigma}$ of $\Sigma$ such that $(\tilde{\Omega},\tilde{\Sigma},\mu)$ is a $\sigma$-finite measure, $L_1(\tilde{\Omega},\tilde{\Sigma},\mu)$ is separable and\\ $M\subset L_1(\tilde{\Omega},\tilde{\Sigma},\mu)$. Further $L_1(\tilde{\Omega},\tilde{\Sigma},\mu)$ is $T_n$-invariant for each $n\in\mathbb{N}$.
\end{lemma}
\begin{proof} Replace the set $C_k$ defined as in the proof of Lemma 1, \cite{L.T.} by $C_{k}:=\{T_{m}\chi_{B} : B\in B_{k}, m\in\mathbb{N}\}$, where $\{T_m \}_{m\in\mathbb{N}}$ is a countable family of operators in $L_{1}(\Omega,\Sigma,\mu)$ and $B_k$ are similar defined as in Lemma 1 of \cite{L.T.}.
\end{proof}
%
%
\begin{proposition} $(i)$ Consider $L_1(\Omega,\Sigma,\mu)$, where $(\Omega,\Sigma,\mu)$ is an arbitrary measure space. Let $\{T_n\}_{n\in\mathbb{N}}$ be a countable set of operators in $\partial\Phi(L_1(\Omega,\Sigma,\mu))$.
Then there exist $\tilde{\Omega}\in\Sigma$ and a $\sigma$-subring $\tilde{\Sigma}$ of $\Sigma$, such that $(\tilde{\Omega},\tilde{\Sigma},\mu)$ is a $\sigma$-finite measure space, $E:=L_{1}(\tilde{\Omega},\tilde{\Sigma},\mu)$ is separable and $T_n$-invariant. Furthermore we have that
\[T_{n}|_{E}\in\partial\Phi(E)\]
holds for each $n\in\mathbb{N}$, where $T_n|_{E}:E\rightarrow E$ is the restriction of $T_n$ to $E$. As a conclusion each $T_n$ is not tauberian.\\
$(ii)$ Let $T\in L(L_1(\Omega,\Sigma,\mu))$. Then there exists a $T$-invariant subspace $E$ isomorphic to $l^1$ or $L_1[0,1]$ such that
\[\partial\sigma_{ess}(T)\subset\partial\sigma_{ess}(T|_{E})\subset\{\lambda\in\mathbb{C}\  |\  T-\lambda \text{ is not tauberian }\}.\]
\end{proposition}
\begin{proof} Consider $T_{n}\in\partial\Phi(X)$, where $X=L_1(\Omega,\Sigma,\mu)$. Note that in particular $T_n\notin\Phi_{+}(X)$ for each $n$, (see \cite{V.M.}, p. 169, Lemma 1). Hence there exists a sequence $(f_{n,m})_{m\in\mathbb{N}}$ in $X$ such that $T_nf_{n,m}$ is convergent but $f_{n,m}$ has no convergent subsequence. Define $M:=\overline{\text{span}}\{f_{n,m} | n,m\in\mathbb{N}\}$. Further there exists  sequence $T_{n,m}\in\Phi(X)$ such that $T_{n,m}$ converges to $T_{n}$ for $m\rightarrow\infty$. Since $T_{n,m}\in\Phi(X)$ we find $S_{n,m}\in\Phi(X)$ and $L_{n,m},\ K_{n,m}\in K(X)$
such that 
\[S_{n,m}T_{n,m}=I+K_{n,m} \quad\text{ and }\quad T_{n,m}S_{n,m}=I+L_{n,m}\]
holds for all $n,m\in\mathbb{N}$. By the above Lemma we find a separable subspace $E:=L_1(\tilde{\Omega},\tilde{\Sigma},\mu)$ which contains $M$ and is invariant for $\{S_{n,m}\}_{n,m\in\mathbb{N}}\cup\{T_{n,m}\}_{n,m\in\mathbb{N}}$. Consider the restrictions $T_{n,m}|_E$ and $S_{n,m}|_E$. Then we get that
\[S_{n,m}|_E\ T_{n,m}|_E=I_E+K_{n,m}|_E \quad\text{ and }\quad T_{n,m}|_{E}\  S_{n}|_E =I_E+L_{n,m}|_{E}\]
holds for all $n,m$ and $I_E$ denotes the identity on $E$. It follows that $T_{n,m}|_E\in\Phi(E)$ and converges to $T_{n}|_{E}$ for $m\rightarrow\infty$. Since $M\subset E$ it follows that $T_{n}|_{E}$ is not upper-semi Fredholm and therefore $T_{n}|_{E}\in\partial\Phi(E)$ for each $n\in\mathbb{N}$. Note that since $E$ is separable and infinite dimensional ($M\subset E$), it follows that $E$ is isomorphic to $l_1(\mathbb{N})$ or $L_1[0,1]$, (see \cite{P.W.}, p. 83) . By Lemma \ref{lemma-Weis} we conclude that $T_{n}|_{E}$ is not tauberian and therefore $T_n$. To prove $(ii)$ take any sequence $\{\lambda_n\}$ dense in $\partial\sigma_{ess}(T)$. Then apply $(i)$ to the operators $T_n:=T-\lambda_n$.
\end{proof}

The next proposition can be found in \cite{Bermudez}.
%
%
\begin{proposition}\label{Grothendieck} Suppose $X$ is a Grothendieck space and $T\in L(X)$. Then $N(T^{*})=\{0\}$ holds if and only if $N(T^{***})=\{0\}$ holds.
\end{proposition}
%
%
\begin{lemma}(\textbf{Spectral lemma}, \cite{A.N.}, \cite{G.C.})\label{spectral-lemma}
Let $X$ be a Banach space and consider $T\in L(X)$.
\begin{itemize}\item[\textnormal{(i)}] Assume there exists a vector $x\in X$ such that $J_{T}(x)$ has non-empty interior. Then for every
$\lambda\in\mathbb{C}$ with $|\lambda|\leq 1$ the operator $T-\lambda I$  has dense range. This is equivalent to say that \[\sigma_{p}(T^{*})\cap\overline{\mathbb{D}}=\emptyset.\]
\item[\textnormal{(ii)}] Assume there exists a non-zero vector $x\in X$ such that $J_{T}(x)$ has non-empty interior. Then \[\sigma(T)\cap\partial\mathbb{D}\neq\emptyset.\]
\end{itemize} 
\end{lemma}
%
%
\begin{theorem}\label{C(K) spectrum} Consider $C(K)$ where $K$ is a compact Hausdorff space. Then the following assertions are equivalent.
\begin{itemize}
\item[\textnormal{(i)}] $C(K)$ is a Grothendieck space (e.g. $l^{\infty}$ or $l^{\infty}/c_0$).
\item[\textnormal{(ii)}] For every $T\in L(C(K))$ we have
\[\partial\sigma(T)\subset\sigma_{p}(T^{*}).\]
\end{itemize}
In both cases we have that \[\overline{\mathbb{D}}\subset\sigma(T)\] holds for every operator $T\in L(C(K))$ with $(J_{T}(x))^{\circ}\neq\emptyset$ for some $x\neq 0$. In particular there does not exist an invertible $J$-class operator.
\end{theorem}
\begin{proof} $(i) \Rightarrow (ii)$ Note that $(C(K))^{*}$ is isomorphic to an $AL$-space and hence can be written  as $L_1(\Omega,\Sigma,\mu)$ for some measure space $(\Omega,\Sigma,\mu)$, see \cite{Abramovich}, p. 95 . Suppose first that $\lambda\in\partial\sigma(T)$ and $\lambda\in\sigma_{ess}(T)$. Then $\lambda\in\partial\sigma_{ess}(T)=\partial\sigma_{ess}(T^{*})$ holds and therefore by the above theorem we know that $S^{*}$ is not tauberian, where $S:=T-\lambda I$. Hence $N(S^{***})\neq N(S^{*})$ since $S^{*}(B_{X^{*}})$ is always closed, see \cite{M.G.}, p. 12.  Furthermore since $C(K)$ has the Grothendieck property by assumption, it follows by Proposition \ref{Grothendieck} that $N(S^{*})\neq \{0\}$ holds. If  $\lambda\in\partial\sigma(T)=\partial\sigma(T^{*})$ but not in $\sigma_{ess}(T)=\sigma_{ess}(T^{*})$ then it is isolated and an eigenvalue of $T^{*}$, see \cite{Abramovich}, p. 300 /301. To show that $(ii)\Rightarrow (i)$, consider first an arbitrary Banach space $X$ with the property that every $T\in L(X)$ satisfies $\partial\sigma(T)\subset\sigma_{p}(T^{*})$. We will show that $X$ has no  complemented separable subspaces. Assume the contrary and decompose $X=M\oplus N$, where $M$ is separable. Choose now a bounded sequence $(x_n, x^{*}_n)\subset M\times M^*$ with the property that  the linear span of $x_n$ lies dense in $M$ and $x^*_n(x_m)=\delta_{n,m}$. Define now the compact operator on $M$ via $K(x):=\sum_{n=1}^{\infty}\frac{1}{2^n}x_n^{*}(x)x_n$. Then $\sigma(K)=\{0\}\cup\{\lambda_k\}_{k\in\mathbb{N}}$ ($\lambda_k$ are possible eigenvalues of $K$), but $0\notin\sigma_{p}(K^{*})$ since $K$ has dense range. Consider the operator $T$ on $X$ defined by $T=K\oplus I$, where $I$ is the identity on $N$. Then we have $\partial\sigma(T)=\sigma(T)=\{0,1\}\cup\{\lambda_{k}\}_{k\in\mathbb{N}}$ and on the other side $0\notin\sigma_{p}(K^{*}\oplus I)=\sigma_p(T^{*})$, a contradiction. In particular in the case where $X=C(K)$ any copy of $c_0$ cannot be complemented. By Cor. 2 in \cite{P.C.} it follows that $C(K)$ has the Grothendieck property.\\
Assume now that there exists a vector $x\in C(K)$, such that $(J_{T}(x))^{\circ}\neq\emptyset$ and $C(K)$ has the Grothendieck property. By the Spectral Lemma we have $\sigma_{p}(T^{*})\cap\overline{\mathbb{D}}=\emptyset\quad (*)$ and $\sigma(T)\cap\partial\mathbb{D}\neq\emptyset\quad (**)$. Hence we get from the above and $(*)$ that $\partial\sigma(T)\cap\overline{\mathbb{D}}=\emptyset$ holds and because of $(**)$ and an easy connectedness argument we conclude that $\overline{\mathbb{D}}\subset\sigma(T)$.%
\end{proof}
%
%
\section{Characterization of $J$-Class operators which are adjoints} 
\begin{definition} Let $X$ be a Banach space and $T\in L(X)$. The \textit{injectivity modulus} (sometimes called the minimum modulus) is defined as
\[\kappa(T):=\inf\{\left\|Tx\right\|:x\in X, \left\|x\right\|=1\}.\]
We define further
\[i(T):=\lim\limits_{n\rightarrow\infty}\kappa(T^{n})^{\frac{1}{n}}.\]
The \textit{surjectivity modulus} of $T$ is defined as 
\[s(T):= \sup\{ r\geq0: T(B_{X})\supset r\cdot B_{X}\},\]
where $B_{X}$ denotes the closed unit ball.
Moreover we define
\[\delta(T):=\lim\limits_{n\rightarrow\infty}s(T^{n})^{\frac{1}{n}}.\]
\end{definition}
%
%
\begin{remark}\label{value-remark} The following inequalities hold \[\kappa(ST)\geq\kappa(S)\kappa(T) \text{ and } s(ST)\geq s(S)s(T),\]  
see  \cite{V.M.}, p. 82.
\end{remark}
%
%
For the set $L_{A}^{W}(l^{\infty}):=\{T\in L(l^{\infty})\ |\  T=S^{*}+W, \text{ where } S\in L(l^1)\text{ and } W\in L(l^\infty) \text{ weakly compact}\}$ we have the following lemma stated and proved in \cite{A.N.2}. 
\begin{lemma} \cite{A.N.2} \label{spectral-lemma2} Let $T\in L_{A}^{W}(l^\infty)$. Then we have that
\[\sigma_{a}(T^{*})=\sigma_{p}(T^{*})\] holds. If in addition $T$ satisfies the statement of the Spectral Lemma, then $\overline{\mathbb{D}}\subset\sigma_{p}(T)$ holds.
\end{lemma}
%
%
\begin{proposition}\label{pre-main-equivalence} Consider the operator $T\in L_{A}^{W}(l^{\infty})$. Then the following statements are equivalent.
\begin{itemize}
\item[\textnormal{(i)}] $T$ is $J^{mix}$-class.
\item[\textnormal{(ii)}] $T$ is $J$-class.
\item[\textnormal{(iii)}] $J_{T}(x)$ has non-empty interior for some $x\neq0$.
\item[\textnormal{(iv)}] $\sigma(T)\cap\partial\mathbb{D}\neq\emptyset$ and $\sigma_{p}(T^{*})\cap\overline{\mathbb{D}}=\emptyset$. 
\item[\textnormal{(v)}] $i(T^{*})>1$ and $\overline{\mathbb{D}}\subseteq\sigma_{p}(T)$.
\item[\textnormal{(vi)}] $i(T^{*})>1$ and $0\in\sigma_{p}(T)$.
\end{itemize}
\end{proposition}
\begin{proof}The implications $(i)\Rightarrow (ii)\Rightarrow (iii)$ are clear and $(iii)\Rightarrow (iv)$ is the statement of the Spectral Lemma. The direction $(v)\Rightarrow (vi)$ is trivial. 
Let us prove that $(iii)$ implies $(iv)$. First observe that by Lemma \ref{spectral-lemma2}  and the Spectral Lemma we get that \[\sigma_{a}(T^{*})\cap\overline{\mathbb{D}}=\sigma_{p}(T^{*})\cap\overline{\mathbb{D}}=\emptyset\text{ and } \overline{\mathbb{D}}\subseteq\sigma_{p}(T).\] So we have just to show that $i(T^{*})>1$. Now for any bounded operator $R$ on a complex Banach space $X$ we have that 
\[\text{dist}(\{0\},\sigma_{a}(R))=i(R)\]
holds, see \cite{V.M.}, p. 91. Since $\sigma_a(T^{*})$ is closed this implies
\[1<\text{dist}(0,\sigma_{a}(T^{*}))=i(T^{*}).\] 
This shows that $(iv)$ implies $(v)$. Assume that $(vi)$ holds. For arbitrarily operators $T$ on Banach spaces we have the following relation:
\[i(T^{*})=\lim\limits_{n\rightarrow\infty} \kappa((T^{*})^{n})^{\frac{1}{n}}=\lim\limits_{n\rightarrow\infty} s((T)^{n})^{\frac{1}{n}}=\delta(T),\]
where we used the fact that $\kappa(T^{*})=s(T)$, see \cite{V.M.}, \text{ p. 83}.\\
With our assumption that $i(T^{*})>1$ we can therefore find $n_{0}$ large enough and $\varepsilon>0$, such that 
\[s(T^{n_0})>(1+\varepsilon)^{n_0}.\]

Choose at this point any $y\in l^{\infty}\backslash\{0\}$ and consider $\tilde{y}:=\frac{y}{\left\|y\right\|}$. Looking at the definition of the surjectivity modulus, we can therefore find $x\in l^{\infty}$ with $\left\|x\right\|\leq 1$, such 
that
\[T^{n_0}x=(1+\varepsilon)^{n_0}\tilde{y},\]
or equivalently
\[T^{n_0}\tilde{x}=y,\]
where 
\[\tilde{x}=\frac{\left\|y\right\|}{(1+\varepsilon)^{n_0}}\cdot x\ \text{ and so }\ \left\|\tilde{x}\right\|\leq\frac{\left\|y\right\|}{(1+\varepsilon)^{n_0}}.\]
Define $x_1:=\tilde{x}$, and like above we can find $x_{2}$ with
\[T^{n_0}x_{2}=x_1\ \text{ and }\  \left\|x_{2}\right\|\leq\frac{1}{(1+\varepsilon)^{n_0}}\left\|x_1\right\|.\]
Hence 
\[T^{2n_0}x_2=y\ \text{ and }\  \left\|x_2\right\|\leq \frac{1}{(1+\varepsilon)^{2n_0}}\left\|y\right\|.\]
Inductively we are able to find a sequence $(x_m)_{m}$ in $l^{\infty}$ with the property that
\[T^{mn_0}x_m=y\  \text{ and }\  \left\|x_m\right\|\leq\frac{1}{(1+\varepsilon)^{mn_0}}.\]
With $z_m:=T^{n_0}x_m$ we get in particular \[\lim\limits_{m\rightarrow\infty}z_{m}=0\  \text{ and }\  \lim\limits_{m\rightarrow\infty}T^{m}z_m=y.\] Since $y$ was arbitrarily it follows that $J_{T}^{mix}(0)=l^{\infty}$. Last but not least, take into consideration that $N(T)\neq\{0\}$, which implies together with $J_{T}^{mix}(0)=l^{\infty}$ that there exists some $x\neq 0$, such that $J_{T}^{mix}(x)=l^{\infty}$. This shows $(i)$.
\end{proof}
%
%
\begin{remark} \label{Remarkable-Remark}a) Take into consideration that there are operators for which there exists a non-zero vector $x$ such that $J_{T}(x)$ has non-empty interior, but is not $J$-class, see \cite{Azimi}.\\
$b)$ It follows from the proof of Proposition \ref{pre-main-equivalence} that the direction $(vi)\Rightarrow (i)$ is also valid for arbitrary Banach spaces.
\end{remark}
At this point we will consider for the sake of simplicity those $T\in L(l^{\infty})$ which are adjoints of operators on $l^1$, i.e. $T=S^{*}$, where $S\in L({l^1})$. Of course these $T$ are contained in $L_{A}^{W}(l^\infty)$. Moreover take into consideration that $i(S)=\delta(S^{*})=\delta(T)=i(T^{*})$ which can be easily seen as in the above theorem, since $\kappa(T^{*})=s(T)$ and $\kappa(T)=s(T^{*})$ (see \cite{V.M.}, p. 83). Hence as a direct consequence of Proposition \ref{pre-main-equivalence}  we get the following important corollary:
%
%
\begin{corollary}\label{main equivalence} Consider the operator $T\in L(l^{\infty})$ which is an adjoint operator, i.e. there exist some $S\in L(l^1)$ with $T=S^{*}$. Then the following statements are equivalent:
\begin{itemize}
\item[\textnormal{(i)}] $T$ is $J^{mix}$-class.
\item[\textnormal{(ii)}] $T$ is $J$-class.
\item[\textnormal{(iii)}] $J_{T}(x)$ has non-empty interior for some non-zero vector $x$.
\item[\textnormal{(iv)}] $\sigma(T)\cap\partial\mathbb{D}\neq\emptyset$ and $\sigma_{p}(T^{*})\cap\overline{\mathbb{D}}=\emptyset$. 
\item[\textnormal{(v)}] $i(S)>1$ and $\overline{\mathbb{D}}\subseteq\sigma_{p}(T)$.
\item[\textnormal{(vi)}] $i(S)>1$ and $0\in\sigma_{p}(T)$.
\end{itemize}
\end{corollary}
%
%
\begin{corollary} Let $T_{1},T_{2}\in L(l^\infty)$ be commuting adjoint operators, which are also $J$-class. Then $T:=T_{1}T_{2}$ is also $J$-class.
\end{corollary}
\begin{proof} Let $S_1$ and $S_2$ be the corresponding operators on $l^1$, such that $(S_i)^*=T_i$ for $i\in\{1,2\}$. We have $T=S^{*}_{1}S^{*}_{2}=(S_{2}S_{1})^{*}$ and by the above theorem, we get that $i(S_i)>1$ holds for $i=1,2$. Furthermore we get from Remark \ref{value-remark}
\[
\kappa((S_{2}S_{1})^n)=\kappa(S_{2}^{n}S_{1}^{n})\geq\kappa(S_2^n)\kappa(S_1^n)
\]
for all $n\in\mathbb{N}$. Hence $i(S_{2}S_{1})\geq i(S_2)\cdot i(S_1)>1$ and $N(T_{1}T_{2})\supset N(T_2)\neq\{0\}$. This shows that condition (v) of Theorem \ref{main equivalence} is satisfied for $T$ and is therefore $J$-class.
\end{proof}
%
%
\begin{proposition}\label{open} Let $X$ be a Banach space. Then the set
\[M:=\{T\in L(X): \delta(T)>1 \text{ and } N(T)\neq\{0\}\}\]
is open with respect to the norm-topology.
\end{proposition}
\begin{proof} Consider any $T\in M$. Then there exist some $m\in\mathbb{N}$ and a positive constant such that $s(T^m)>c>1$. In particular $T$ is surjective and $N(T)\neq\{0\}$. Define 
\[g: L(X)\rightarrow L(X^{*})\  \text{ by }\  g(S):=(S^{*})^{m}.\]
Then $g$ is continuous since the mappings $S\rightarrow S^{*}$ and $S\rightarrow S^m$ are continuous with respect to the operator norm-topologies. Now the set of all surjective operators with non-trivial kernel is open with respect to the norm-topology, see \cite{Abramovich}, p. 73. For $\varepsilon:=\frac{c-1}{2}$ we can therefore choose $\nu>0$ small enough such that
for all $S\in L(X)$ satisfying $\left\|S-T\right\|<\nu$ we have that $S$ is surjective, $N(S)\neq\{0\}$ and \[\left\|(T^*)^{m}-(S^{*})^{m}\right\|<\varepsilon.\] Hence it follows
\[\left\|(T^*)^{m}x^{*}\right\|<\left\|(S^*)^{m}x^{*}\right\|+\varepsilon \text{ for all } x^{*}\in X^{*} \text{ with } \left\|x^{*}\right\|=1\]
and therefore
\[c<s(T^{m})=\kappa((T^{*})^{m})\leq \kappa((S^{*})^{m})+\varepsilon=s(S^{m})+\varepsilon.\]
This implies in particular $1<\frac{c+1}{2}=c-\varepsilon<s(S^{m})$ and therefore by Remark \ref{value-remark}
\[\left(\frac{c+1}{2}\right)^{n}<(s(S^{m}))^{n}\leq s(S^{mn}).\]
Hence we get 
\[1<\left(\frac{c+1}{2}\right)^{\frac{1}{m}}<(s(S^{mn}))^{\frac{1}{mn}}\rightarrow \delta(S) \text{ for } n\rightarrow\infty.\]
It follows that the open ball with radius $\nu$ and center $T$ is included in $M$.
\end{proof}
%
%
\begin{corollary} The set 
\begin{align*}JA(l^1)&:=\{S\in L(l^1): S^{*} \text{ is } \text{J-class}\}\\
&=\{S\in L(l^1): \sigma(S)\cap\partial\mathbb{D}\neq\emptyset,\overline{\mathbb{D}}\cap\sigma_{p}(S^{**})=\emptyset\}
\end{align*}
is open with respect to the norm-topology in $L(l^1)$.
\end{corollary}
\begin{proof}Consider the linear isometry $\phi:L(l^1)\rightarrow L(l^{\infty})$ defined by $\phi(S)=S^{*}$. Then we have
\begin{align*}\phi^{-1}(M)&=\{S\in L(l^1): \delta(S^{*})>1 \text{ and } N(S^*)\neq\emptyset\}\\
&=\{S\in L(l^1): i(S)>1 \text{ and } N(S^*)\neq\emptyset\}\\
&=\{S\in L(l^1): S^{*} \text{ is } \text{J-class}\}\\
&=\{S\in L(l^1): \sigma(S^{*})\cap\partial\mathbb{D}\neq\emptyset, \overline{\mathbb{D}}\cap\sigma_{p}(S^{**})=\emptyset)\}\\
&=\{S\in L(l^1): \sigma(S)\cap\partial\mathbb{D}\neq\emptyset, \overline{\mathbb{D}}\cap\sigma_{p}(S^{**})=\emptyset\},
\end{align*}
where $M$ is the set as in Proposition \ref{open} applied to $l^\infty$ and since $\phi$ is continuous the desired statement follows by Proposition \ref{open} and the equivalences of Corollary \ref{main equivalence}.
\end{proof}
\begin{remark} The latter corollary shows the strong relation between adjoint $J$-class operators on $l^{\infty}$ and their spectral behavior. It is also interesting in its own regarding operators on $l^1$ and their spectrum. 
\end{remark}
As promised before we will now characterize a large class of operators on $l^\infty$ using our new tools established above.
%
%
\begin{theorem}\label{$f$-theorem} (i) Consider $T\in L_{A}^{W}(l^{\infty})$ with the corresponding operators $S\in L(l^1)$ and $W\in L(l^\infty)$. Let $f$ be a holomorphic map on a neighborhood of $K_{R}(0)$, where $R>\max\{r(S^{*}+W),r(S^{*})\}$ and $K_{R}(0)$ is the open disk at zero with radius $R$. Then $f(T)$ is $J$-class if and only if
\[f(\sigma_{a}(T^{*}))\cap\overline{\mathbb{D}}=\emptyset \text{ and } \overline{\mathbb{D}}\subseteq f(\sigma(T^{*})\backslash\sigma_{a}(T^{*})).\]
(ii) Let $B_{w}$ be the unilateral backward shift with a positive and bounded weight sequence $w:=(w_n)_{n}$ on $l^{\infty}$ and consider $f(B_{w})$, where $f$ is a holomorphic map defined on a neighborhood of $K_{r_1}$. Then $f(B_{w})$ is $J$-class if and only if \[\overline{\mathbb{D}}\cap f(\overline{K_{r_1}}\backslash K_{r_2})=\emptyset \text{ and } 
\overline{\mathbb{D}}\subseteq f(K_{r_2}),\] where
\begin{align*}
r_1&=\lim\limits_{n\rightarrow\infty}\ \sup\limits_{k\in\mathbb{N}}\ (w_k\cdot\ldots\cdot w_{k+n-1})^{\frac{1}{n}},\\
r_2&=\lim\limits_{n\rightarrow\infty}\ \inf\limits_{k\in\mathbb{N}}\ (w_k\cdot\ldots\cdot w_{k+n-1})^{\frac{1}{n}}
\end{align*}
and $K_{r_i}$ denotes the open disc centered at $0$ with radius $r_i$ for $i=1,2,$.
\end{theorem}
\begin{proof} $(i)$ Let $\sum_{n=0}^{\infty}a_n z^{n}$ be the power series expansion of $f(z)$. For each positive integer we can write $(S^{*}+W)^{n}=(S^{*})^{n}+W_n$, where $W_n$ is weakly compact, since the weakly compact operators form a closed two-sided ideal (see \cite{Abramovich}, p. 89). Since $R>r(S^{*})$ the corresponding operator $f(S^{*})=(f(S))^{*}$ exists and so we get
\begin{align*} f(T)-f(S^{*})&=\sum_{n=0}^{\infty}a_n (S^{*}+W)^{n}-\sum_{n=0}^{\infty}a_n (S^{*})^{n}\\&=
\sum_{n=0}^{\infty}a_n ((S^{*})^{n}+W_{n})-\sum_{n=0}^{\infty}a_n (S^{*})^{n}
=\sum_{n=0}^{\infty}a_n W_{n}=:V
\end{align*}
This shows in particular that the series $V$ exists and is weakly compact. Therefore we get that $f(T)=(f(S))^{*}+V\in L_{A}^{W}(l^{\infty})$ holds. Assume that $f(T)$ is $J$-class. Then by Proposition \ref{pre-main-equivalence}, the Spectral theorem and the Spectral theorem for the approximate spectrum, we get that
\[\emptyset=\overline{\mathbb{D}}\cap\sigma_a(f(T^{*}))=\overline{\mathbb{D}}\cap f(\sigma_a(T^{*}))\tag{$1$}\]
and 
\[\overline{\mathbb{D}}\subseteq\sigma_p(f(T))\subseteq\sigma (f(T))=\sigma((f(T))^{*})=\sigma(f(T^{*}))=f(\sigma(T^{*}))\]
holds. Hence with $(1)$ we get that $\overline{\mathbb{D}}\subseteq f(\sigma(T^{*})\backslash\sigma_{a}(T^{*}))$ holds.
This shows the first direction. For the other direction read the steps in $(1)$ backward which in particular shows that $i(f(T)^{*})>1$. Moreover we have that
\[\overline{\mathbb{D}}\subseteq f(\sigma(T^{*})\backslash\sigma_{a}(T^{*}))=f(\sigma(T)\backslash\sigma_{s}(T))\subseteq f(\sigma_{p}(T))=\sigma_{p}(f(T)),\] 
where we used the Spectral theorem for the point spectrum.
Again by Theorem \ref{pre-main-equivalence} we get that $f(T)$ is $J$-class.\\
$(ii)$ Let $T:=B_{w}$. Consider the forward shift $S_{w}$ on $l^{1}$ with a positive and bounded weight sequence $w=(w_1,w_2,w_3,\ldots,)$, i.e. \[S_{w}(x_1,x_2,x_3,\ldots)=(0,w_{1}x_{1},w_{2}x_2,w_{3}x_3,\ldots).\]
Then $T=S_{w}^{*}$ and the approximate spectrum of $T^{*}$ is exactly (see, \cite{K.L.}, p. 86)
\begin{align*}\sigma_{a}(T^{*})&=\sigma_{a}(S_{w}^{**})=\sigma_a(S_w)=\{\lambda\in\mathbb{C}|\  i(S_w)\leq|\lambda|\leq r(S_w)\}\\
&=\overline{K_{r_1}}\backslash K_{r_2},  
\end{align*}
where 
\[i(S_w)=r_2 \text{ and } r(S_w)=r_1.\]
Further the whole spectrum of $S_w$ is the closed disk with center $0$ and radius $r_1$, see \cite{K.L.}, p. 86. In other words $\sigma(T^{*})=\sigma(S_w)=\overline{K_{r_1}}.$
Hence we get by $(i)$ that $f(B_w)$ is $J$-class if and only if 
\[f(\overline{K_{r_1}}\backslash K_{r_2})\cap\overline{\mathbb{D}}=f(\sigma_{a}(T^{*}))\cap\overline{\mathbb{D}}=\emptyset\text{ and } \overline{\mathbb{D}}\subseteq f(\sigma(T)\backslash\sigma_{a}(T^{*}))=f(K_{r_2}).\]
\end{proof}
The following theorem  gives a characterization of operators of the form $T=f(B)$, whenever they are chaotic. Here $B$ is the unweighted backward shift.
%
%
\begin{theorem}(\cite{Laube}, \cite{Erdmann} p. 122)
Let $X$ be one of the spaces $l^{p}$, $1\leq p<\infty$ or $c_0$. Further let $f$ be a holomorphic function on a neighborhood of $\overline{\mathbb{D}}$. Then the following assertions are equivalent:\\\\
$(i)$ $f(B)$ is chaotic.\\\\
$(ii)$ $f(\mathbb{D})\cap\partial\mathbb{D}\neq\emptyset$.\\\\
$(iii)$ $f(B)$ has a non-trivial periodic point.
\end{theorem}
Referring to the above theorem, we are interested in the case $p=\infty$ and the $J$-class behavior of $f(B)$. Geometrically, condition $(ii)$ is not enough to get the $J$-class property on $l^{\infty}$. More conditions are needed as the next corollary will show.
%
%
\begin{corollary}\label{f(B)-Cor} Consider the backward shift $B$ on $l^{\infty}$ and let $f$ be a holomorphic function in a neighborhood of the closed unit disc. Then the following statements are equivalent.\\\\
(i) $f(B)$ is $J$-class.\\\\
(ii) $f(\partial\mathbb{D})\cap\overline{\mathbb{D}}=\emptyset$ and $\overline{\mathbb{D}}\subseteq f(\mathbb{D})$.
\end{corollary} 
\begin{proof} Note that for the operator $B$ the values $r_1, r_2$ are equal to $1$. Then the statement follows directly from Theorem \ref{$f$-theorem}.
\end{proof}
\begin{example}\textbf{(Costakis, Manoussos, \cite{G.C.2})} Consider $T:=B_w$ on $l^{\infty}$. Then $T$ is $J$-class if and only if 
\[r_2=\lim\limits_{n\rightarrow\infty} \inf\limits_{k\in\mathbb{N}}\ (w_k\cdot w_{k+1}\cdots w_{n+k-1})^{\frac{1}{n}}>1.\]
\end{example}
\begin{proof} Consider $p(z)=z$. Then by Theorem \ref{$f$-theorem} we have that $T$ is $J$-class if and only if \[p(\overline{K_{r_1}}\backslash K_{r_2})\cap\overline{\mathbb{D}}=\overline{K_{r_1}}\backslash K_{r_2}\cap\overline{\mathbb{D}}\overbrace{=}^{(*)}\emptyset\text{ and } \overline{\mathbb{D}}\subset K_{r_2}.\] The last inclusion follows already by $(*)$ and therefore $T$ is $J$-class if and only if 
$r_2>1$. This completes the proof.
\end{proof}
\begin{example} Consider for $m\in\mathbb{N}$ the operator $T:= I+B_{w}^{m}$. Then $T$ is $J$-class if and only if 
\[r_2=\lim\limits_{n\rightarrow\infty} \inf\limits_{k\in\mathbb{N}}\ (w_k\cdot w_{k+1}\cdots w_{n+k-1})^{\frac{1}{n}}>\sqrt[m]{2}.\]
\end{example}

\begin{proof} We can write $T=p(B_w)$, where $p(z)=1+z^m$. If $T$ is $J$-class then 
\begin{align*}1<i(I+S_w^m)&=\text{dist}(\{0\},\sigma_a(I+S^m_w))\\
               &=\min\{|1+z^m|: r_2\leq|z|\leq r_1\}\\
							 &=\min\limits_{\theta\in\left[0,2\pi\right]}|1+r_{2}^{m}e^{i\theta m}|
							 \end{align*}
follows from Theorem \ref{main equivalence}.							
Hence this is equivalent to
\[r_2>2^{\frac{1}{m}}.\] 							
In view of the the definition of $r_2$, the desired direction follows.\\\\
For the other direction take into consideration that
\[r_3\geq r_2>2^{\frac{1}{m}}\] holds, where $r_{3}=\liminf\limits_{n\rightarrow\infty}(w_{1}\cdot\ldots\cdot w_n)^{\frac{1}{n}}$ and it is easy to see that $K_{r_3}\subseteq \sigma_{p}(B_w)$\\ ($K_{r_3}$ is the open disk centered at $0$ with radius $r_3$).\\ This in particular shows that 
\[
\min\limits_{\theta\in\left[0,2\pi\right]}|1+r_{3}^{m}e^{i\theta m}|\geq \min\limits_{\theta\in\left[0,2\pi\right]}|1+r_{2}^{m}e^{i\theta m}|>1
\] 
holds. Therefore we get that $0\in\sigma_{p}(I+B_{w}^{m})$. By Theorem \ref{main equivalence} it follows that $T$ is $J$-class.
\end{proof}
%
%
\begin{proposition}\label{spectrumcontained} Consider the positive weighted unilateral backward shift $B_{w}: l^{\infty}\rightarrow l^\infty$. Then we have
\[\sigma(\widehat{B_w})\subseteq\sigma_a(S_w)=\{\lambda\in\mathbb{C}|\ i(S_w)\leq|\lambda|\leq r(S_w)\},\]
where $S_w$ is the forward shift on $l^{1}$.
\end{proposition}
\begin{proof}
Consider any $\mu\in\mathbb{C}$ with $c:=|\mu|< i(S_w)$. We choose $n_0\in\mathbb{N}$ large enough such that 
\[\inf\limits_{k\in\mathbb{N}}\ w_k\cdot\ldots\cdot w_{k+n_{0}-1}> c^{n_0}.\]
Therefore we can find $\delta<1$ such that 
\[1>\delta>\frac{c^{n_0}}{w_k\cdot\ldots\cdot w_{k+n_{0}-1}}\]
holds for all $k\in\mathbb{N}$. We will show that $(\widehat{B_{w}})^{n_0}-{\mu}^{n_0} I$ is bijective.
Take into consideration that 
\begin{align*}B^{n_0}_{w}(x_1,x_2,\ldots)&=((w_{1} w_{2}\ldots w_{n_0}\cdot x_{{n_{0}}+1}), (w_{2} w_{3}\ldots w_{{n_0}+1}\cdot x_{n_{0}+2}),\ldots)\\
&=(w_{k}w_{k+1}\ldots w_{n_{0}+k-1}\cdot x_{k+n_{0}})_{k\in\mathbb{N}}.
\end{align*}
Consider the equation \[((\widehat{B_{w}})^{n_0}-{\mu}^{n_0} I)\left[x\right]=0 \text{ with } \left[x\right]=\left[(x_{1},x_{2},\ldots)\right],\] which is equivalent to say that 
\[B^{n_0}_{w}(x_1,x_2,\ldots)-{\mu}^{n_0}(x_1,x_2,\ldots)\in c_0,\]
or more precisely
\[w_{k}\cdot\ldots\cdot w_{k+n_{0}-1}x_{k+n_0}-{\mu}^{n_0}x_k=\varepsilon_k \text{ for all }k\geq1,\]
where $\lim\limits_{k\rightarrow\infty}\varepsilon_{k}=0$.
Hence we get the estimate
\begin{align*}|x_{k+n_0}|&\leq\frac{c^{n_0}}{w_{k}\cdot\ldots\cdot w_{k+n_{0}-1}}|x_{k}|+\frac{\varepsilon_k}{w_{k}\cdot\ldots\cdot w_{k+n_{0}-1}}\\
&\leq\delta|x_{k}|+\frac{\varepsilon_k}{{c}^{n_0}}.
\end{align*}
Therefore, 
\begin{align*}
\limsup\limits_{k\rightarrow\infty}|x_{k+n_0}|&\leq\delta\limsup\limits_{k\rightarrow\infty}|x_{k}|+\frac{1}{c^{n_0}}\lim\limits_{k\rightarrow\infty}\varepsilon_k\\
&=\delta\limsup\limits_{k\rightarrow\infty}|x_{k}|.
\end{align*}
Since  $\limsup\limits_{k\rightarrow\infty}|x_{k+n_0}|=\limsup\limits_{k\rightarrow\infty}|x_{k}|$ and $\delta$ is smaller than $1$, it follows that $\limsup\limits_{k\rightarrow\infty}|x_{k}|=0$ and therefore 
$\lim\limits_{k\rightarrow\infty}x_{k}=0$. In other words, $\left[x\right]=0$ holds and hence $(\widehat{B_{w}})^{n_0}-{\mu}^{n_0} I$ is injective.\\\\
Next we show that $(\widehat{B_{w}})^{n_0}-{\mu}^{n_0} I$ is surjective. For that purpose consider the forward shift $S_w$ on $l^{1}$. Since \[\sigma_s(B_{w})=\sigma_{a}(S_{w})=\{\lambda\in\mathbb{C}\ |\ i(S_w)\leq|\lambda|\leq r(S_w)\}\]
it follows that $B_w-e^{i\theta_k}cI$ is surjective for $k=\{1,2,...,n_0\}$, where $e^{i\theta_k}c$ are the roots of the equation $z^{n_0}={\mu}^{n_0}$.
Now notice that 
\[B_w^{n_0}-{\mu}^{n_0}I=(B_{w}-e^{i\theta_1}cI)\cdot\ldots\cdot(B_{w}-e^{i\theta_n}cI)\]
and therefore $B_w^{n_0}-{\mu}^{n_0}I$ is surjective. Since the surjectivity of every operator $T$ implies the surjectivity of the induced operator $\widehat{T}$, it follows that
$(\widehat{B_{w}})^{n_0}-{\mu}^{n_0} I$ is surjective, hence bijective. This means ${\mu}^{n_0}\notin\sigma((\widehat{B_w})^{n_0})$ and therefore $\mu\notin\sigma(\widehat{B_w})$.
From \[\left\|\widehat{B_w}^{n}\right\|\leq\left\|B_{w}^{n}\right\|=\left\|S_{w}^{n}\right\|\]
we get  $r(\widehat{B_{w}})\leq r(S_w)$. Finally we obtain with the preceding calculations that 
\begin{align*}\sigma(\widehat{B_{w}})&\subseteq\{\lambda\in\mathbb{C}\ |\ i(S_w)\leq|\lambda|\leq r(\widehat{B_{w}})\}\\
&\subseteq\{\lambda\in\mathbb{C}\ |\ i(S_w)\leq|\lambda|\leq r(S_w)\}=\sigma_a(S_w)
\end{align*} 
holds. 
\end{proof}
%
%

\begin{lemma}\label{eigenvectorlemma} Let $(w_n)_{n}$ be a positive and bounded sequence. Suppose there exists $\lambda_{0}\in\mathbb{C}$ and $\varepsilon>0$ such that the open disk $K_{\varepsilon}(\lambda_0)$ is contained in $K_{m}(0)$, where 
\[m:=\lim\limits_{n\rightarrow\infty}\inf\ (\prod_{k=1}^{n}w_k)^{\frac{1}{n}}>0.\]
Then the span of the set consisting of the vectors
\[e_{\lambda}:=(\frac{\lambda}{w_1},\frac{\lambda^2}{w_{1}w_{2}},\ldots),\]
with $\lambda\in K_{\varepsilon}(\lambda_0)$ is dense in $c_0$.
\end{lemma}
\begin{proof}
Consider any $y^{*}\in {c_0}^{*}\cong l^1$. Assume $y^{*}(e_\lambda)=0$. We can find a sequence $(y_n)_{n}\in l^1$ which can be uniquely identified with $y^{*}$. That means
\[y^{*}(e_{\lambda})=\sum_{n=0}^{\infty}y_n\cdot\frac{\lambda^{n}}{(\prod_{k=1}^{n}w_k)}\]
holds.
The above equation defines a power series $\sum_{n=0}^{\infty}a_n \lambda^{n}$ on $K_m(0)$ with
$a_n:=\frac{y_n}{(\prod_{k=1}^{n}w_k)}$, which is identically zero  for all 
$\lambda\in K_{\varepsilon}(\lambda_0)$. Therefore $a_n=0$ for all $n\in\mathbb{N}$, which in turn implies that $y_n=0$ for all $n\in\mathbb{N}$. This shows that $y^{*}=0$. By a well-known application of the Hahn-Banach theorem it follows that the span of the $e_{\lambda}$ is dense in $c_0$.
\end{proof}
%
%
\begin{lemma}\label{induced J-class}
Let $X$ be a Banach space and $T\in L(X)$. Assume that $T$ is $J$-class and that there exists a $T$-invariant and closed subspace $M\subseteq A_T$ such that $A_T\backslash M\neq\emptyset$. Then $\widehat{T}$ is $J$-class.
\end{lemma}
\begin{proof} We choose any vector $x\in A_T\backslash M$, then $\left[x\right]\neq 0$. Now consider any $\left[y\right]\in X/M$.
Then there exists a strictly increasing sequence $(n_k)_{k}$ of positive integers and a sequence $(x_k)_{k}$ with $x_k\rightarrow x$, such that $T^{n_k}x_k\rightarrow y$. Hence $\left[x_k\right]\rightarrow \left[x\right]$ and $\widehat{T}^{n_k}\left[x_k\right]=\left[T^{n_k}x_k\right]\rightarrow\left[y\right]$ as $k\rightarrow\infty$ and this shows that $\widehat{T}$ is $J$-class.
\end{proof}
%
%
\begin{theorem}\label{J-set of f(B_w)} Consider the weighted backward shift on $l^{\infty}$. Assume that $T:=f(B_w)$ is $J$-class, where $f$ is a holomorphic function defined on a open neighborhood of $\sigma(B_w)$. Then 
\[A_{T}=c_0.\] 
\end{theorem}
\begin{proof}
To show that $A_{T}\subseteq c_0$, we consider the induced operator $\widehat{T}:l^{\infty}/c_0\rightarrow l^{\infty}/c_0$ and assume that $A_{T}\backslash c_0\neq\emptyset$. By Lemma \ref{induced J-class}, the operator $\widehat{T}$ is $J$-class. Then we get from the Spectral mapping theorems (see \cite{Erdmann}, p. 365 and \cite{P.A.}, p. 83) and Proposition \ref{spectrumcontained}
\[\sigma(\widehat{T})=\sigma(\widehat{f(B_w))})=f(\sigma(\widehat{B_w}))\subseteq f(\sigma_{a}(S_w))=\sigma_{a}(f(S_w)).\tag{$*$}\]
Since $f(B_w)$ is $J$-class it follows as in the proof of Proposition \ref{pre-main-equivalence}, that
\[\sigma_{a}(f(S_w))\cap\overline{\mathbb{D}}=\emptyset\]
holds. Hence by $(*)$ we get that 
\[\sigma(\widehat{T})\cap\overline{\mathbb{D}}=\emptyset\]
holds, which is a contradiction to Lemma \ref{spectral-lemma} (Spectral Lemma), since $\widehat{T}$ is $J$-class.\\\\
To show the other inclusion, take into consideration that $r_2\leq m$ where $r_2$ is the value as in Theorem \ref{$f$-theorem}. Hence by the same theorem we get that $\mathbb{D}\subseteq\overline{\mathbb{D}}\subseteq f(K_{r_2}(0))\subseteq f(K_m(0))$ holds, where $K_{r_2}(0)$ is the open disk with center zero and radius $r_2$. The set $f^{-1}(\mathbb{D})\cap K_m(0)$ is open and non-empty. Choose at this point any $\lambda_0$ and $\varepsilon>0$ such that $K_{\varepsilon}(\lambda_0)\subseteq f^{-1}(\mathbb{D})\cap K_m(0)$. Then, $f(K_{\varepsilon}(\lambda_0))\subseteq f(f^{-1}(\mathbb{D})\cap K_m(0))\subseteq \mathbb{D}$. Now each $e_\lambda$, defined as in Lemma \ref{eigenvectorlemma} is an eigenvector of $B_w$ to the corresponding eigenvalue $\lambda\in K_{\varepsilon}(\lambda_0)$. Hence, $e_\lambda$ is an eigenvalue of $f(B_w)$ to the corresponding eigenvalue $f(\lambda)\in f(K_{\varepsilon}(\lambda_0))$. Since $|f(\lambda)|<1$, it follows that $e_\lambda$ is a $J$-vector by \cite{G.C.}, 5.9 . By Lemma \ref{eigenvectorlemma} the set
$\{e_\lambda:\lambda\in K_{\varepsilon}(\lambda_0)\}$ is dense in $c_0$.
This shows $c_0\subseteq A_T$, since $A_T$ is closed by \cite{G.C.}, 2.12. 
\end{proof}
%
%
\section{Open Problems} We conclude this work with some open problems which occurred during this research.\\\\
\textbf{Problem 1.} We have seen in Section 2 that the boundary of the spectrum of an operator on a $C(K)$-space with the Grothendieck property is contained in the point spectrum of its adjoint. The boundary of the spectrum is especially contained in the surjectivity spectrum. So one can formulate the more general question: Is it true that an operator $T:C(K)\rightarrow C(K)$ with dense range, where $C(K)$ has the Grothendieck property, is actually surjective?\\\\
In view of Theorem \ref{C(K) spectrum} and the fact that $C(K)$ has the Dunford Pettis- and the Grothendieck property we state the following problem.\\\\ 
\textbf{Problem 2.} Suppose $X$ is a Banach space which has the property that $\partial\sigma(T)\subset\sigma_{p}(T^{*})$ for every operator $T\in L(X)$. Does this imply that $X$ has the Grothendieck- and the Dunford Pettis property, or at least one of these properties?\\\\
\textbf{Problem 3.} Does Corollary 2.6 hold also for the space $H^{\infty}$, i.e. the space of analytic and bounded functions defined on $\mathbb{D}$?

\end{document}